\documentclass[12pt,3p]{amsart}
\usepackage{bbm}
\usepackage{amsmath,amssymb, amsthm}
\usepackage[colorlinks,urlcolor=red]{hyperref}
\usepackage{graphicx, enumerate}
\theoremstyle{plain}
\newtheorem{theorem}{Theorem}[section]

\newtheorem{lemma}[theorem]{Lemma}
\newtheorem{proposition}[theorem]{Proposition}
\numberwithin{equation}{section}\theoremstyle{definition}
 \newtheorem{definition}[theorem]{Definition}

\theoremstyle{remark}

\allowdisplaybreaks

\numberwithin{equation}{section}

\begin{document}

\title[A fixed point theorem for monotone asymptotic nonexpansive Mappings ]
{A fixed point theorem for monotone asymptotic nonexpansive Mappings }
\author[M. R. Alfuraidan \& M. A. Khamsi]{Monther Rashed Alfuraidan,   Mohamed Amine Khamsi}

\maketitle

\vspace*{-0.5cm}


\begin{center}
{\footnotesize  Department of Mathematics \& Statistics,
King Fahd University of Petroleum and Minerals\\ Dhahran 31261, Saudi Arabia\\ monther@kfupm.edu.sa, \\ Department of Mathematical Sciences, University of Texas at El Paso\\ El Paso, TX 79968, USA \\ mohamed@utep.edu}
\end{center}

\hrulefill

{\footnotesize \noindent {\bf Abstract.}
Let $C$ be a nonempty, bounded, closed, and convex subset of a Banach space $X$ and $T: C \rightarrow C$ be a monotone asymptotic nonexpansive mapping.  In this paper, we investigate the existence of fixed points of $T$.  In particular, we establish an analogue to the original Goebel and Kirk's fixed point theorem for asymptotic nonexpansive mappings.

\vskip0.5cm


}

\hrulefill

\section{Introduction}

Recently a new direction has been discovered dealing with the extension of the Banach Contraction Principle \cite{banach} to partially ordered metric spaces. Ran and Reurings \cite{RR} successfully carried such attempt while investigating the solution(s) to the matrix equation:
$$X = Q \pm \sum_{i = 1}^{i=m} A^*_i F(X) A_{i},$$
where $X \in H(n)$, the set of $n\times n$ Hermitian matrices, $F: H(n) \rightarrow H(n)$ is a monotone function, i.e., $F(X_1) \leq F(X_2)$ if $X_1 \leq X_2$, which maps the set of all $n\times n$ positive definite matrices $P(n)$ into itself, $A_1, \dots, A_m$ are arbitrary $n\times n$ matrices and $Q\in P(n)$, a result known before to Turinici \cite{turinici}.  Another similar approach was carried in \cite{NRL} with applications to some differential equations. Jachymski \cite{jachymski} as the first to give a more general unified version of these extensions by considering graphs instead of a partial order.  In all these works, the mappings considered are monotone contractions.  The case of monotone nonexpansive mappings was first considered in \cite{bachar-amine}.  Then the race was on to find out whether the classical fixed point theorems for nonexpansive mappings still hold for monotone nonexpansive mappings.  In particular, an analogue to Browder \cite{browder} and G\"{o}hde \cite{gohde} fixed point theorems for monotone mappings does hold \cite{buthinah-amine}.  But it is still unknown whether an analogue to the classical Kirk's fixed point theorem \cite{kirk65} holds for monotone nonexpansive mappings.  The difficulty in doing this resides in the fact that the monotone Lipschitzian mappings enjoy nice properties only on comparable elements.  In fact, they may not be even continuous, a property obviously shared by Lipschitzian mappings.  In this paper, we extend Goebel and Kirk's fixed point theorem \cite{goebel-kirk} for asymptotically nonexpansive mappings to the case of monotone mappings.\\

An interesting reference with many applications of the fixed point theory of monotone mappings is the excellent book by Carl and Heikkil\"{a} \cite{ch}.

\section{Preliminaries}

Let $(M,d)$ be a metric space endowed with a partial order $\preceq$.  We will say that $x, y \in M$ are comparable whenever $x \preceq y$ or $y \preceq x$.  Next we give the definition of monotone mappings.

\begin{definition} Let $(M,d, \preceq)$ be a metric space endowed with a partial order.  Let $T: M \rightarrow M$ be a map. $T$ is said to be monotone or order-preserving if
$$x \preceq y \Longrightarrow T(x) \preceq T(y),$$
for every $x,y \in M$.
\end{definition}

Next we give the definition of monotone Lipschitzian mappings.\\

\begin{definition}\label{def-asymptotic-ne} Let $(M,d, \preceq)$ be a metric space endowed with a partial order.  Let $T: M \rightarrow M$ be a map. $T$ is said to be
monotone Lipschitzian mapping if $T$ is monotone and there exists $k \geq 0$ such that
$$d(T(x),T(y)) \leq k \ d(x,y),$$
for every $x,y \in M$ such that $x$ and $y$ are comparable.  We will say that $T$ is a monotone asymptotic nonexpansive mapping if there exists $\{k_n\}$ a sequence of positive numbers such that $\lim\limits_{n \rightarrow +\infty} \ k_n =1$ and
$$d(T^n(x),T^n(y)) \leq k_n \ d(x,y),$$
for every comparable elements $x,y \in M$. A point $x \in M$ is said to be a fixed point of $T$ whenever $T(x) = x$.  The set of fixed points of $T$ will be denoted by $Fix(T)$.
\end{definition}

Note that monotone Lipschitzian mappings are not necessarily continuous.  They usually have a good topological behavior on comparable elements but not on the entire set on which they are defined.

\medskip
\noindent Before we close this section, recall that a sequence $\{x_n\}_{n \in \mathbb{N}}$ in a partially ordered set $(M, \preceq)$ is said to be
\begin{enumerate}
\item[(i)] monotone increasing if $x_n \preceq x_{n+1}$, for every $n \in \mathbb{N}$;
\item[(ii)] monotone decreasing if $x_{n+1} \preceq x_n$, for every $n \in \mathbb{N}$;
\item[(iii)] a monotone sequence if it is either monotone increasing or decreasing.
\end{enumerate}

\section{Monotone Asymptotic Nonexpansive Mappings}
The fixed point theory for asymptotic nonexpansie mappings finds its root in the work of Goebel and Kirk \cite{goebel-kirk}.  Following some successful results on monotone mappings in recent years, the poriginal fixed point theorem of Goebel and Kirk for these mappings was elusive till now.  The setting will be uniformly convex Banach spaces partially ordered.  \\

\begin{definition} \label{UC_def_gen}
Let $(X,\|.\|)$ be a Banach space. We say that $X$ is uniformly convex (in short, UC) if for every $\varepsilon >0$
$$\delta(\varepsilon) = \inf \Big\{1 - \left\|\frac{x+y}{2}\right\|; \|x\| \leq 1,\ \|y\|\leq 1,\ \|x-y\| \geq \varepsilon \Big\} > 0.$$
the function $\delta$ is known as the modulus of uniform convexity of $X$.
\end{definition}

\medskip
\noindent The following technical lemma will be useful to the proof of our main result.

\begin{lemma}\label{technical-lemma-type}  Let $C$ be a nonempty closed convex subset of uniformly convex Banach space $(X,\|.\|)$.  Let $\tau: C \rightarrow [0,+\infty)$ be a type function, i.e., there exists a bounded sequence $\{x_n\} \in X$ such that
$$\tau(x) = \limsup_{n \rightarrow +\infty} \|x_n - x\|,$$
for every $x \in C$.  Then $\tau$ has a unique minimum point $z \in C$ such that
$$\tau(z) = \inf \{\tau(x);\ x \in C\} = \tau_0.$$
Moreover, if $\{z_n\}$ is a minimizing sequence in $C$, i.e. $\lim\limits_{n \rightarrow +\infty} \tau(z_n) = \tau_0$, then $\{z_n\}$ converges strongly to $z$.
\end{lemma}
\begin{proof}  Note that $\tau$ is continuous and convex.  Let us show the existence of the minimum point of $\tau$.  For every $n \geq 1$, the subset $C_n = \{x \in C;\ \tau(x) \leq \tau_0 + 1/n \}$ is not empty and is a closed convex subset of $C$ which is bounded.  The reflexivity of $X$ implies that $C_\infty = \bigcap\limits_{n \geq 1} C_n \neq \emptyset$.  Clearly we have $C_\infty = \{z \in C;\ \tau(z) = \tau_0\}$.  Let us prove that $C_\infty$ is reduced to one point.  Let $z_1$ and $z_2$ be in $C_\infty$.  Assume that $z_1 \neq z_2$.  In this case, we must have $\tau_0 \neq 0$.  Let $\alpha \in (0, \tau_0)$.  Then there exists $n_0 \geq 1$ such that for every $n \geq n_0$, we have
$$\|x_n- z_1\| \leq \tau_0 + \alpha,\; and\;  \|x_n- z_2\| \leq \tau_0 + \alpha.$$
Set $\displaystyle \varepsilon = \frac{\|z_1-z_2\|}{2\tau_0}$, then
$$\left\|x_n - \frac{z_1+z_2}{2} \right\| = \left\|\frac{(x_n -z_1)+(x_n-z_2)}{2} \right\| \leq (\tau_0 +\alpha) \left(1 - \delta(\varepsilon)\right),$$
for every $n \geq n_0$, which implies
$$\tau\left(\frac{z_1+z_2}{2}\right) \leq (\tau_0 +\alpha) \left(1 - \delta(\varepsilon)\right).$$
Hence $\tau_0 \leq (\tau_0 +\alpha) \left(1 - \delta(\varepsilon)\right)$, for every $\alpha \in (0,\tau_0)$.  If we let $\alpha \rightarrow 0$, we will get $\tau_0 \leq \tau_0 \ (1 - \delta(\varepsilon)) < \tau_0$.  This contradiction implies that $C_\infty$ is reduced to one point, i.e. $\tau$ has a unique minimum point.  Next, let $\{z_n\}$ be a minimizing sequence of $\tau$.  Let us prove that $\{z_n\}$ converges to the minimum point $z$.  This conclusion is obvious if $\tau_0 = 0$.  Assume $\tau_0 > 0$ and $\{z_n\}$ does not converge to $z$.  Since $\{x_n\}$ is bounded, then $\{z_n\}$ is also bounded.  Therefore, there exists $R > 0$ such that
$$\max\Big(\|x_n-z_m\|, \|x_n -z\|\Big) \leq R,$$
for every $n, m \in \mathbb{N}$.  Since $\{z_n\}$ does not converge to $z$, we may assume that
$$\varepsilon = \inf \left\{\frac{\|z_m - z\|}{R}; \ m \in \mathbb{N}\right\} > 0.$$
Using the definition of the modulus of convexity of $X$, we get
$$\left\|x_n - \frac{z_m+z}{2} \right\| = \left\|\frac{(x_n -z_m)+(x_n-z)}{2} \right\| \leq \max\Big(\|x_n-z_m\|, \|x_n -z\|\Big) \left(1 - \delta(\varepsilon)\right),$$
for every $n, m \in \mathbb{N}$.  If we let $n \rightarrow +\infty$, taking the limit-sup, we get
$$\tau\left(\frac{z_m+z}{2} \right) \leq \max\Big(\tau(z_m), \tau(z)\Big) \left(1 - \delta(\varepsilon)\right),$$
for every $m \in \mathbb{N}$, which implies
$$\tau_0 \leq \tau(z_m)\ \left(1 - \delta(\varepsilon)\right).$$
If we let $m \rightarrow +\infty$, we get $\tau_0 \leq \tau_0\ \left(1 - \delta(\varepsilon)\right) < \tau_0$.  This contradiction implies that $\{z_n\}$ does converge to $z$.  The proof of Lemma \ref{technical-lemma-type} is complete.
\end{proof}

\medskip

\noindent Since the main result of this work is set in a partially ordered Banach space, we assume that $(X, \|.\|)$ is endowed with a partial order $\preceq$.  Throughout, we assume that order intervals are convex and closed.  Recall that an order interval is any of the subsets
$$[a,\rightarrow) = \{x \in X; a \preceq x\},\; and \; (\leftarrow,b] = \{x \in X; x \preceq b\},$$
for every $a,b \in X$.

\medskip

\noindent Now we are ready to state the main result of this work.

\begin{theorem}\label{monotone-asymptotic-ne-uc}  Let $(X, \|.\|, \preceq)$ be a partially ordered Banach space for which order intervals are convex and closed.  Assume $(X, \|.\|)$ is uniformly convex.  Let $C$ be a nonempty convex closed bounded subset of $X$ not reduced to one point.  Let $T: C \rightarrow C$ be a continuous monotone asymptotic nonexpansive mapping.  Then $T$ has a fixed point if and only if there exists $x_0 \in C$ such that $x_0$ and $T(x_0)$ are comparable.
\end{theorem}
\begin{proof}  Obviously if $x$ is a fixed point of $T$, then $x$ and $T(x) = x$ are comparable.  Let $x_0 \in C$ be such that $x_0$ and $T(x_0)$ are comparable.  Without loss of any generality, assume that $x_0 \preceq T(x_0)$.  Since $T$ is monotone, then we have $T^n(x_0) \preceq T^{n+1}(x_0)$, for every $n \in \mathbb{N}$.  In other words, the orbit $\{T^n(x_0)\}$ is monotone increasing.  Since the order intervals are closed and convex and $X$ is reflexive, we conclude that
$$C_\infty = \bigcap\limits_{n \geq 0} [T^n(x_0), \rightarrow)\cap C = \bigcap\limits_{n \geq 0} \{x \in C;\ T^n(x_0) \preceq x\} \neq \emptyset.$$
Let $x \in C_\infty$, then $T^n(x_0) \preceq x$ and since $T$ is monotone, we get
$$T^n(x_0) \preceq T(T^n(x_0)) = T^{n+1}(x_0) \preceq T(x),$$
for every $n \geq 0$, i.e., $T(C_\infty) \subset C_\infty$.  Consider the type function $\tau: C_\infty \rightarrow [0,+\infty)$ generated by $\{T^n(x_0)\}$, i.e. $\tau(x) = \limsup\limits_{n \rightarrow +\infty} \|T^n(x_0)-x\|$.  The Lemma \ref{technical-lemma-type} implies the existence of a unique $z \in C_\infty$ such that $\tau(z) = \inf \{\tau(x);\ x \in C_\infty\} = \tau_0$.  Since $z \in C_\infty$, we have $T^p(z) \in C_\infty$, for every $p \in \mathbb{N}$, which implies
$$\tau(T^p(z)) = \limsup\limits_{n \rightarrow +\infty} d(T^n(x_0),T^p(z)) \leq k_p\ \limsup\limits_{n \rightarrow +\infty} d(T^n(x_0),z),$$
where $\{k_p\}_{p \in \mathbb{N}}$ is given by Definition \ref{def-asymptotic-ne} such that $\lim\limits_{p \rightarrow +\infty} k_p = 1 $ since $T$ is asymptotically nonexpansive.  Hence $\tau_0 \leq \tau(T^p(z)) \leq k_p \ \tau_0$, for every $p \in \mathbb{N}$.  The main property of $\{k_p\}_{p \in \mathbb{N}}$ implies
$$\lim\limits_{p \rightarrow +\infty}  \tau(T^p(z)) = \tau_0, $$
which means that $\{T^p(z)\}_{p \in \mathbb{N}}$ is a minimizing sequence of $\tau$.  Using Lemma \ref{technical-lemma-type} again, we conclude that  $\{T^p(z)\}_{p \in \mathbb{N}}$ converges to $z$.  Since $T$ is continuous, we have $\lim\limits_{p \rightarrow +\infty} T(T^p(z)) = \lim\limits_{p \rightarrow +\infty} T^{p+1}(z) = T(z) = z$, i.e. $z$ is a fixed point of $T$.
\end{proof}

\medskip
\noindent It is natural to ask whether the continuity assumption in Theorem \ref{monotone-asymptotic-ne-uc} may be relaxed.  This is the main motivation behind \cite{NRL} where the authors relaxed the continuity assumption from the main result of \cite{RR}.  Looking at the proof carefully, we see that the continuity assumption was used at the end to prove that the minimum point is a fixed point.  The difficulty met here has to do with the fact that it is not clear whether the minimum point is comparable to its image under the map in question.  While investigating this point, we came with a property satisfied by any Banach lattice, like the classical $L^p([0,1])$-spaces (for $p\geq 1$), similar to the Opial condition \cite{opial}.  It is well known that the classical $\ell^p$ spaces (for $p \geq 1$) enjoy the Opial condition for the weak topology while $L^p([0,1])$-spaces (for $p > 1$) fail to enjoy such property despite the fact that these spaces are uniformly convexity.

\begin{definition}\label{opial}  Let $(X,\|.\|, \preceq)$ be a partially ordered Banach space.
\begin{enumerate}
\item[(i)] \cite{opial}  $X$ is said to satisfy the weak-Opial condition if whenever any sequence $\{x_n\}$ in $X$ which weakly converges to $x$, we have
$$\limsup_{n \rightarrow +\infty} \|x_n - x\| < \limsup_{n \rightarrow +\infty} \|x_n - y\|,$$
for every $y \in X$ such that $x \neq y$.
\item[(ii)] $X$ is said to satisfy the monotone weak-Opial condition if whenever any monotone increasing (resp. decreasing) sequence $\{x_n\}$ in $X$ which weakly converges to $x$, we have
$$\limsup_{n \rightarrow +\infty} \|x_n - x\| \leq \limsup_{n \rightarrow +\infty} \|x_n - y\|,$$
for every $y \in X$ such that $x \preceq y$ (resp. $y \preceq x$).
\end{enumerate}
\end{definition}

\medskip
\noindent The following result is amazing.

\begin{proposition}\label{lattice-opial}  Any Banach lattice satisfies the monotone weak-Opial condition.
\end{proposition}
\begin{proof}  Let $(X, \|.\|, \preceq)$ be a Banach lattice.  One of the properties of $X$ states
$$0 \preceq u \preceq v \;\; implies\;\; \|u\| \leq \|v\|,$$
for every $u,v \in X$, and the positive cone $P$ of $X$ is convex and closed  \cite{LT}.   In order to show that $X$  satisfies the monotone weak-Opial condition, let $\{x_n\}_{n \in \mathbb{N}}$ be a monotone sequence in $X$ which converges weakly to $x$.  Without loss of generality, we assume that $\{x_n\}_{n \in \mathbb{N}}$ is monotone increasing.  Let $y \in X$ be such that $x \preceq y$.  Since order intervals of $X$ are convex and closed, we conclude that $x_n \preceq x \preceq y$ which implies
$0 \preceq x-x_n \preceq y-x_n$, for every $n \in \mathbb{N}$.  Hence  $\|x-x_n\| \leq \|y-x_n\|$, for every $n \in \mathbb{N}$, which implies
$$\limsup_{n \rightarrow +\infty} \|x_n - x\| \leq \limsup_{n \rightarrow +\infty} \|x_n - y\|.$$
\end{proof}

\medskip\medskip
\noindent Using the monotone Opial condition, we can relax the continuity assumption in Theorem \ref{monotone-asymptotic-ne-uc}.

\begin{theorem}\label{monotone-asymptotic-ne-uc-opial}  Let $(X, \|.\|, \preceq)$ be a partially ordered Banach space for which order intervals are convex and closed.  Assume $(X, \|.\|)$ is uniformly convex.  Let $C$ be a nonempty convex closed bounded subset of $X$ not reduced to one point.  Assume that $X$ satisfies the monotone weak-Opial condition.  Let $T: C \rightarrow C$ be a monotone asymptotic nonexpansive mapping.  Then $T$ has a fixed point if and only if there exists $x_0 \in C$ such that $x_0$ and $T(x_0)$ are comparable.
\end{theorem}
\begin{proof} As we did in the proof of Theorem \ref{monotone-asymptotic-ne-uc}, we first assumed that $x_0 \preceq T(x_0)$.  Then the orbit $\{T^n(x_0)\}$ is a monotone increasing sequence.  Since $X$ is reflexive, it is easy to show that $\{T^n(x_0)\}$ is weakly convergent to some point $x \in C$ and $T^n(x_0) \preceq x$, for every $n \in \mathbb{N}$.  Since $X$ satisfies the monotone weak-Opial condition, we know that
$$\limsup_{n \rightarrow +\infty} \|T^n(x_0) - x\| \leq \limsup_{n \rightarrow +\infty} \|T^n(x_0) - y\|,$$
for every $y \in \widetilde{C} = C \cap [x, \rightarrow)$.  Note that $\widetilde{C}$ is a nonempty closed convex subset of $C$.  Therefore the type function $\tau$ generated by the orbit $\{T^n(x_0)\}$ has $x$ as its unique minimum point in $\widetilde{C}$.  As we did in the proof of Theorem \ref{monotone-asymptotic-ne-uc}, we see that $\{T^p(x)\}$ converges strongly to $x$.  Let us show that $x \preceq T(x)$.  We have $T^n(x_0) \in (\leftarrow, x]\cap C$, for every $n \in \mathbb{N}$.  Since $T$ is monotone we get $T^{n+1}(x_0) \in (\leftarrow, T(x)]\cap C$, for every $n \in \mathbb{N}$.  Since $(\leftarrow, T(x)]\cap C$ is convex and closed, then the weak-limit of $\{T^{n+1}(x_0)\}$ also belong to this set, i.e. $x \in (\leftarrow, x]\cap C$.  In other words, we have $x \preceq T(x)$.  This will imply that $\{T^n(x)\}$ is a monotone increasing sequence which converges to $x$.  Therefore we must have $T^n(x) \preceq x$, for every $n \in \mathbb{N}$.  This will force $x = T(x)$, i.e. $x$ is a fixed point of $T$.
\end{proof}

\medskip
\noindent  {\bf Acknowledgment}

\medskip
\noindent The authors would like to acknowledge the support provided by the deanship of scientific research at King Fahd University of Petroleum \& Minerals for funding this work through project No. IP142-MATH-111.


\bigskip



\begin{thebibliography}{999}

\bibitem{bachar-amine}
        M. Bachar and M. A. Khamsi
        \emph{Fixed points of monotone mappings and application to integral equations},
        Fixed Point Theory and Applications 2015, 2015:110. DOI:10.1186/s13663-015-0362-x.

\bibitem{banach}
        S. Banach,
        \emph{Sur les op\'{e}rations dans les ensembles abstraits et leurs applications},
        Fund. Math. 3(1922), 133-181.

\bibitem{buthinah-amine}
        B. A. Bin Dehaish, M A. Khamsi,
        \emph{Browder and Gohde fixed point theorem for monotone nonexpansive Mappings},
        Fixed Point Theory and Applications 2016:20 DOI: 10.1186/s13663-016-0505-8

\bibitem{browder}
        F. E. Browder,
        \emph{Nonexpansive nonlinear operators in a Banach space},
        Proc. Nat. Acad. Sci. U.S.A., \textbf{54}~(1965), 1041-1044.

\bibitem{ch}
    S. Carl, S. Heikkil\"{a},
    \emph{Fixed Point Theory in Ordered Sets and Applications: From Differential and Integral Equations to Game Theory,}
    Springer, Berlin, New York, 2011.

\bibitem{goebel-kirk}
        K. Goebel and W. A. Kirk,
        \emph{A fixed point theorem for asymptotically nonexpansive mappings},
        Proc. Amer. Math. Soc. 35 (1972), 171–-174.

\bibitem{gohde}
        D. G\"{o}hde,
        \emph{Zum Prinzip der kontraktiven Abbildung,} Math. Nachr. \textbf{30}~(1965), 251-258.

\bibitem{jachymski}
        J. Jachymski,
        \emph{The Contraction Principle for Mappings on a Metric Space with a Graph,} Proc. Amer. Math. Soc. 136(2008), 1359--1373.

\bibitem{kirk65}
        W. A. Kirk,
        \emph{A fixed point theorem for mappings which do not increase distances},
        Amer. Math. Monthly 72(1965), 1004-1006.

\bibitem{LT}
    J Lindenstrauss, L Tzafriri
    \emph{Classical Banach Spaces II}, Springer-Verlag, New York/Berlin (1979).

\bibitem{NRL}
        J. J. Nieto, R. Rodriguez-Lopez,
        \emph{Contractive mapping theorems in partially ordered sets and applications to ordinary differential equations},
        Order 22 (2005), no. 3, 223--239.

\bibitem{opial}
    Z. Opial,
    \emph{Weak convergence of the sequence of successive approximations for nonexpansive mappings,}
    Bull. Amer. Math. Soc., 73 (1967), pp. 591–-597

\bibitem{RR}
    A. C. M. Ran, M. C. B. Reurings, \emph{A fixed point theorem in partially ordered sets and some applications to matrix equations}, Proc. Amer. Math. Soc. 132 (2004), no. 5, 1435--1443.


\bibitem{turinici}
    M. Turinici,
    \emph{Fixed points for monotone iteratively local contractions},
    Dem. Math., 19 (1986), 171-180.

\end{thebibliography}
\end{document}